\documentclass{amsart}

\usepackage{amsmath}
\usepackage{amsthm}
\usepackage{amsfonts}
\usepackage{amssymb}
\usepackage{mathabx}

\usepackage{pgf,tikz}
\usepackage{mathrsfs}
\usetikzlibrary{arrows}
\usepackage[mathscr]{euscript}

\usepackage{graphics}
\usepackage{epsfig}
\usepackage{color}
\usepackage{verbatim}

\usepackage{float}

\usepackage[utf8]{inputenc}
\usepackage{enumerate}
\usepackage[bookmarks = true, colorlinks=true, linkcolor = black, citecolor = black, menucolor = black, urlcolor = black]{hyperref}

\newcommand\R{{\mathbb{R}}}

\newcommand\C{{\mathbb{C}}}
\newcommand\Z{{\mathbb{Z}}}
\newcommand\N{{\mathbb{N}}}

\newcommand\dist{\operatorname{dist}}

\DeclareMathOperator*{\esssup}{ess\,sup}

\newcommand{\wt}{\widetilde}

\newcommand{\Norm}[1]{ \left\|  #1 \right\| }

\newcommand\CA{{\mathcal A}}

\newcommand\CM{{\mathcal M}}

\newcommand\CS{{\mathcal S}}

\theoremstyle{plain}
  \newtheorem{theorem}{Theorem}[section]

  \newtheorem{lemma}[theorem]{Lemma}
  \newtheorem{corollary}[theorem]{Corollary}

\theoremstyle{definition}
  \newtheorem{definition}[theorem]{Definition}
  
  \newtheorem{remark}[theorem]{Remark}

\date{}
\title{Sparse bounds for pseudodifferential operators}
\author{David Beltran}
\address{David Beltran: Basque Center for Applied Mathematics (BCAM), Alameda de Mazarredo 14, 48009 Bilbao, Spain}
\email{dbeltran@bcamath.org}

\author{Laura Cladek}
\address{Laura Cladek: Department of Mathematics, University of California, Los Angeles, CA 90095}
\email{cladekl@math.ucla.edu}

\begin{document}

\maketitle

\begin{abstract}

We prove sparse bounds for pseudodifferential operators associated to Hörmander symbol classes. Our sparse bounds are sharp up to the endpoint and rely on a single scale analysis. As a consequence, we deduce a range of weighted estimates for pseudodifferential operators. The results naturally apply to the context of oscillatory Fourier multipliers, with applications to dispersive equations and oscillatory convolution kernels.



\end{abstract}

\section{Introduction and statement of results}

Given a smooth function $a \in C^\infty(\R^n \times \R^n)$, define the associated pseudodifferential operator $T_a$ by
\begin{align*}\label{opdef}
T_af(x)=\int_{\mathbb{R}^n}e^{ix\cdot\xi}a(x,\xi)\widehat{f}(\xi)d\xi,
\end{align*}
where $f \in \mathcal{S}$ and $\widehat{f}$ denotes the Fourier transform of $f$. The smooth function $a$ is assumed to belong to the symbol classes $S^m_{\rho,\delta}$, introduced by Hörmander in \cite{Hormander} and further studied by many authors, see for instance \cite{CV1972, Fe73, bigStein, CSpdos}. Recall that $S^m_{\rho,\delta}$ consists of all $a \in C^\infty(\R^n \times \R^n)$ satisfying the differential inequalities
\begin{equation}\label{symbol}
|\partial_x^\nu \partial_\xi^\sigma a(x,\xi)| \lesssim (1+|\xi|)^{m-\rho|\sigma|+\delta|\nu|}
\end{equation}
for all multi-indices $\nu,\sigma \in \N^n$, where $m \in \R$ and $0\leq \delta, \rho \leq 1$. \footnote{We use the notation $A \lesssim B$ to denote that there is a constant $C$ such that $A \leq CB$ and the notation $A \lesssim_{\epsilon} B$  to specify the dependence of the implicit constant in a certain parameter $\epsilon$. We omit the constant factors of $\pi$ coming from our normalisation of the Fourier transform.}
\newline
\indent
The main goal of this article is to prove sparse domination for such pseudodifferential operators. Obtaining sparse bounds for classical operators in harmonic analysis has been an active area of research in recent years, starting with fundamental work of Lerner \cite{LeCZ, LeA2} in the context of Calderón--Zygmund operators. The original Banach space norm domination of Lerner was shortly further refined to a pointwise sparse control, see \cite{CAR2014, LN, Lac2015, LeNew}. More recently, a new perspective on the subject has been taken, which consists in a domination in the context of bilinear forms. This was introduced by Bernicot, Frey and Petermichl \cite{BFP} and was further developed in the work of Culiuc, Di Plinio and Ou \cite{CDO1,CDO2}. This approach has proved to be highly successful, as it applies to operators that fall well beyond the classical Calderón--Zygmund theory. Among many examples, we may find Bochner--Riesz multipliers \cite{BBL, LMR}, rough singular integrals \cite{CACDO}, the bilinear Hilbert transform \cite{CDO1}, the variational Carleson operator \cite{DDG}, oscillatory singular integrals \cite{LS,KLSparse}, spherical maximal functions \cite{LaceySpherical}, a specific singular Radon transform \cite{OberlinSparse} or the recent work by Ou and the second author \cite{CO} for Hilbert transforms along curves; we also refer to  \cite{KLdiscrete, CKL} for sparse domination in a discrete setting. 
\newline
\indent
While the technique of sparse domination has been successfully applied to study weighted bounds for a wide variety of operators in harmonic analysis, as the ones listed above, it has not previously been used to yield any weighted bounds for pseudodifferential operators. However, weighted $A_p$ estimates for these objects have been studied before. Chanillo and Torchinksy proved in \cite{CT} that if $a \in S^{-n(1-\rho)/2}_{\rho,\delta}$, with $0<\delta<\rho<1$, then $T_a$ is bounded on $L^p(w)$ for $w \in A_{p/2}$ and $2 \leq p < \infty$. Later, Michalowski, Rule and Staubach \cite{MRS10} extended this result to $\delta=\rho$, and also proved that for the smaller classes $S^{-n(1-\rho)}_{\rho,\delta}$, with $0<\rho \leq 1$, $0 \leq \delta<1$, there is $L^p(w)$ boundedness for $w \in A_p$ and $1<p<\infty$, see \cite{MRScan}. More recently, the first author \cite{Bel2016} established weighted boundedness for the symbol classes $S^m_{\rho,\delta}$ with $-n(1-\rho)/2<m<0$ and $0 \leq \delta \leq \rho < 1$; in this case $T_a$ is bounded on $L^p(w)$ if $w \in A_{p/2} \cap RH_{(2t'/p)'}$ and $2 \leq p < 2t'$, where $t=(\rho-1)n/2m$. \footnote{Given $1 \leq r \leq \infty$, $r'$ denotes its conjugate Hölder exponent, that is $1/r+1+r'=1$.}
\newline
\indent
The weighted bounds obtained as a corollary of our sparse domination theorem yield new weighted estimates for pseudodifferential operators and also recover all the bounds mentioned in the above paragraph, except for some endpoint cases. A more detailed discussion is postponed to Section \ref{sec:weighted}. The simplicity of the methods through which we are able to obtain weighted corollaries that improve existing results in the literature is a further testament to the power and versatility of the method of sparse domination. 
\newline
\indent
We first give the definition of sparse collections of cubes and sparse bilinear forms, and then follow with the statement of our main theorem. It is instructive to interpret our results recalling the embeddings of the symbol classes, that is $S_{\rho_1,\delta_1}^{m_1} \subseteq S_{\rho_2,\delta_2}^{m_2}$ if $m_1 \leq m_2,  \rho_1 \geq \rho_2$ and $\delta_1 \leq \delta_2$.
\begin{definition}
A collection of dyadic cubes $\mathcal{S}$ in $\mathbb{R}^n$ is $\eta$-sparse for $0<\eta<1$ if there exist sets $\{E_{{Q}}: Q\in\mathcal{S}\}$ that are pairwise disjoint, $E_Q\subset Q$, and satisfy $|E_Q|>\eta|Q|$ for all $Q\in\mathcal{S}$. For any dyadic cube $Q$, for $1\le p<\infty$ define $\left<f\right>_{p,Q}^p:=|Q|^{-1}\int_Q|f|^p\,dx$. Given $1\le r, s<\infty$, the $(r, s)$ sparse form $\Lambda_{\mathcal{S}, r, s}$ is defined as
\begin{align*}
\Lambda_{\mathcal{S}, r, s}(f, g):=\sum_{Q\in\mathcal{S}}|Q|\left<f\right>_{r,Q}\left<g\right>_{s,Q},
\end{align*}
and the $r$-sparse operator $\CA_{r,\CS}$ is defined as
$$
\mathcal{A}_{r,\mathcal{S}}f(x):=\sum_{Q \in \mathcal{S}} \left< f \right>_{r,Q} \chi_Q(x).
$$
\end{definition}

\begin{theorem}\label{mainps}
Let $a \in S_{\rho, \delta}^m$ for some $m<0$ and $0<\delta\le\rho<1$. Then for any compactly supported bounded functions $f, g$ on $\mathbb{R}^n$, there exist sparse collections $\mathcal{S}$ and $\widetilde{\mathcal{S}}$ of dyadic cubes such that
\begin{align*}
|\left<T_af, g\right>|\le C(m,\rho, r, s)\Lambda_{\mathcal{S}, r, s'}(f, g)
\end{align*}
and
\begin{align*}
|\left<T_af, g\right>|\le C(m,\rho, r, s)\Lambda_{\widetilde{\mathcal{S}}, s', r}(f, g)
\end{align*}
for all pairs $(r, s')$ and $(s', r)$ such that 
\begin{align*}
m<-n(1-\rho)(1/r-1/2), \qquad 1\le r\le s \le 2
\end{align*}
or
\begin{align*}
m<-n(1-\rho)(1/r-1/s), \qquad 1\le r\le 2 \le s \le r'.
\end{align*}
\end{theorem}
The sparse bounds in Theorem \ref{mainps} are best understood in the plane $(1/r,1/s')$. Given a symbol $a \in S^m_{\rho,\delta}$, where $m<0$ and $0 \leq \delta \leq \rho < 1$, the pairs $(r,s')$ for which a sparse bound holds satisfy that $(1/r,1/s')$ lie inside an \textit{open} trapezoid $\mathcal{T}_{m,\rho,n}$ with vertexes depending on the parameters $m, \rho$ and the dimension $n$ (see Figure \ref{sparse region}). Such vertexes are given by
$$
v_1=(1,0), \quad v_2=\Big(1, \frac{-m}{n(1-\rho)}\Big), \quad v_3=\Big(\frac{-m}{n(1-\rho)}, 1\Big), \quad v_4=(0,1),
$$
if $-n(1-\rho) \leq m\leq -n(1-\rho)/2$, and

\begin{align*}
&v_1=\Big(\frac{1}{2} + \frac{-m}{n(1-\rho)}, \frac{1}{2} + \frac{m}{n(1-\rho)}\Big),  && v_2 =\Big(\frac{1}{2} + \frac{-m}{n(1-\rho)}, 1/2\Big),
\\
&v_3=\Big(\frac{1}{2}, \frac{1}{2} + \frac{-m}{n(1-\rho)}\Big),  && v_4=\Big(\frac{1}{2} + \frac{m}{n(1-\rho)}, \frac{1}{2} + \frac{-m}{n(1-\rho)}\Big),
\end{align*}
if $ -n(1-\rho)/2<m<0$.

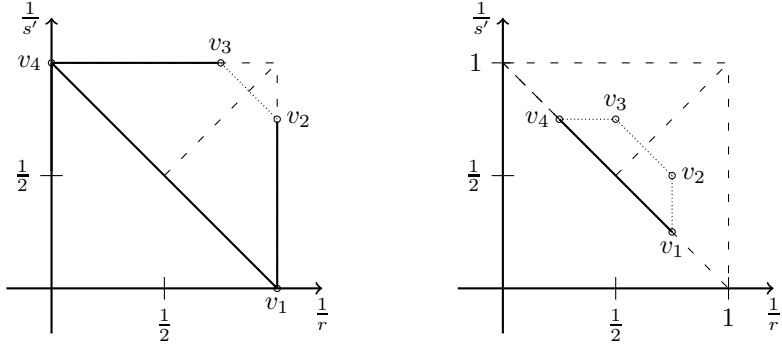
\begin{figure}[H]
\begin{tikzpicture}[scale=2] 

\begin{scope}[scale=1.5]
\draw[thick,->] (-.2,0) -- (1.2,0) node[below] {$ \frac 1 r$};
\draw[thick,->] (0,-.2) -- (0,1.2) node[left] {$ \frac{1}{s'}$};

\draw[loosely dashed] (0,1) -- (1.,1.)  -- (1.,0); 
\draw[thick] (0,1) -- (0.5,0.5)  -- (1.,0);

\draw (.5,.05) -- (.5,-.05) node[below] {$ \tfrac 12$};
\draw (.05,.5) -- (-.05,.5) node[left] {$ \tfrac 12$};

\draw (1,0) circle (.04em) node[below] {$ v_1$};  
\draw (0,1) circle (.04em) node[left] {$ v_4$}; 

\draw[thick] (0,1) -- (0,0.52);

\draw[loosely dashed] (0,1) -- (0.5,0.5)  -- (1.,0); 

\draw[loosely dashed]  (0.5,0.5)  -- (1.,1); 

\draw[thick] (0,1) -- (0.74,1);
\draw[thick] (1,0) -- (1,.74);

\draw[densely dotted] (1,0.75) -- (.75,1);

\draw (1.,.75) circle (.04em) node[right] { $ v_2$};
\draw (.75,1) circle (.04em) node[above] { $ v_3$};

\begin{scope}[xshift=2cm] 

\draw[thick,->] (-.2,0) -- (1.2,0) node[below] {$ \frac 1 r$};
\draw[thick,->] (0,-.2) -- (0,1.2) node[left] {$ \frac{1}{s'}$};

\draw[loosely dashed] (0,1) -- (1.,1.)  -- (1.,0); 
\draw[loosely dashed] (0,1) -- (1/4,3/4);
\draw[loosely dashed] (3/4,1/4) -- (0,1);
\draw[thick] (1/4,3/4) -- (0.5,0.5)  -- (3/4.,1/4);

\draw (.5,.05) -- (.5,-.05) node[below] {$ \tfrac 12$};
\draw (.05,.5) -- (-.05,.5) node[left] {$ \tfrac 12$};

\draw (1,.05) -- (1,-.05) node[below] {$ 1$};
\draw (.05,1) -- (-.05,1) node[left] {$ 1$};

\draw (3/4,1/4) circle (.04em) node[below] {$ v_1$};  
\draw (1/4,3/4) circle (.04em) node[left] {$ v_4$}; 

\draw[loosely dashed] (0,1) -- (0.5,0.5)  -- (1.,0); 

\draw[loosely dashed]  (0.5,0.5)  -- (1.,1);

\draw[densely dotted] (0.75,1/4) -- (.75,1/2);
\draw[densely dotted] (1/4,3/4) -- (1/2,3/4);
\draw[densely dotted] (0.75,1/2) -- (.5,3/4);

\draw (3/4,.5) circle (.04em) node[right] { $ v_2$};
\draw (.5,3/4) circle (.04em) node[above] { $ v_3$};

\end{scope}
\end{scope}

\end{tikzpicture}

\caption{The region for sparse bounds. These bounds are sharp up to the boundary of the region, and also include the thick boundary segments. The trapezoid $\mathcal{T}_{m,\rho,n}$ on the left corresponds to the case $-n(1-\rho)<m \leq -n(1-\rho)/2$, and the one on the right to the case $-n(1-\rho)/2 < m < 0$.}
\label{sparse region}

\end{figure}

The methods to prove the above domination theorem are inspired by the recent works of Lacey and Spencer \cite{LS} and Lacey, Mena and Reguera \cite{LMR}. It consists in obtaining geometrically decaying sparse bounds for a single scale version of the operators under study. Its short proof will be given in Section \ref{sec:proof}.
\newline
\indent
The above sparse bounds are sharp up to the endpoints, as it will be shown in Section \ref{sec:endpoints}. Concerning endpoint results, we remark that the symbols in the class $S^0_{1,\delta}$ with $\delta<1$ \footnote{The case $\delta=1$ is naturally excluded from all statements as it is well know that there are symbols in that class that fail to be bounded on $L^2$; see, for instance \cite{bigStein}.} are Calderón--Zygmund operators, and thus a pointwise sparse bound by $L^1$ averages (in the context of the upcoming Theorem \ref{pointwise}) follows from \cite{CAR2014,LN}. Also, for the specific classes $S^{-n(1-\rho)}_{\rho,\delta}$, with $0<\rho \leq 1$, $0 \leq \delta<1$, it is possible to obtain the following endpoint sparse bound in the stronger pointwise control context.

\begin{theorem}\label{pointwise}
Let $a \in S^{-n(1-\rho)}_{\rho,\delta}$, $0<\rho\leq 1$, $0\leq \delta<1$. Then for every $f \in C_0^\infty(\R^n)$ and any $r>1$, there exists a sparse family $\CS$ such that for $a.e.$ $x\in \R^n$,
$$
|T_af(x)| \lesssim \mathcal{A}_{r,\mathcal{S}}f(x).
$$
\end{theorem}
Of course this implies domination for the symbol classes $S^{-n(1-\rho)}_{\rho,\delta}$ by a bilinear $(r,1)$ sparse form. The proof of this result is a straightforward consequence of the pointwise domination principle of Lerner \cite{LeNew} and a sharp maximal function estimate in \cite{MRScan}. A sketch of the proof is given in an appendix.
\newline
\indent
Finally, we note that our techniques naturally extend to the setting of oscillatory Fourier multipliers associated to differential inequalities of the type \eqref{symbol}. In particular, we are able to obtain sparse bounds for the model Fourier multipliers $a_{\rho,m}(\xi):=e^{i|\xi|^{1-\rho}}(1+|\xi|)^m \chi_{\{|\xi|^{1-\rho} \geq 1\}}$ for any value of $\rho \in \R$. Interestingly, when $\rho=-1$, the results may be interpreted in the context of the free linear Schrödinger equation. We refer to Section \ref{sec:multipliers} for further discussion.

\subsection*{Acknowledgments.}  The authors would like to thank Michael T. Lacey, Kangwei Li and Maria Reguera for helpful conversations and Virginia Naibo for pointing out an inaccuracy in a previous version of the manuscript. This work was initiated while the authors were in residence at the Mathematical Sciences Research Institute in Berkeley, California, during the Spring 2017 semester. The first author is supported by the ERC Grant 307617,  the ERCEA Advanced Grant 2014 669689 - HADE, the MINECO project MTM2014-53850-P, the Basque Government project IT-641-13, the Basque Government through the BERC 2014-2017 program, by Spanish Ministry of Economy and Competitiveness MINECO: BCAM Severo Ochoa excellence accreditation SEV-2013-0323, and an IMA small grant. The second author is supported by a National Science Foundation Postdoctoral Fellowship, NSF grant 1703715.

\section{Sparse bilinear form domination for Pseudodifferential operators}\label{sec:proof}
This section is devoted to prove Theorem \ref{mainps}, our sparse domination theorem for pseudodifferential operators. To this end, we first prove a domination by a geometrically decaying sequence of sparse bilinear forms, which will be later refined to a domination by a true sparse bilinear form. A key ingredient in the proof is the use of optimal $L^p$ improving estimates for frequency and spatially localised pieces of the operator.
\newline
\indent
Such Lebesgue space estimates are obtained by interpolating first the $L^2\to L^2$ and $L^1\to L^1$ estimates to yield optimal $L^s\to L^s$ bounds for $1 \le s \le 2$. Then, the use of Bernstein's inequality lead to the optimal $L^r\to L^s$ bounds for $1\le r\le s\le 2$, and a further interpolation with the $L^1\to L^\infty$ estimates leads to the optimal $L^r \to L^s$ bounds for $1 \leq r \leq 2 \leq s \leq r'$. 
\newline
\indent
This will allow us to obtain domination by a $(r,s')$ bilinear sparse form for such values of $r$ and $s$. The corresponding $(r,s')$ sparse bounds associated to the remaining possible cases for the values of $r$ and $s$ will follow by duality, as the Hörmander symbol classes of pseudodifferential operators are closed under adjoints.

\subsection*{Decomposition of the operators $T_a$.}
Let $a \in S^m_{\rho, \delta}$, where $m<0$ and $0\leq \delta \leq \rho < 1$, and consider its associated pseudodifferential operator $T_a$. Let $\psi$ be a smooth cutoff that is identically $1$ in the unit ball and supported in its double, and let $\wt{\psi}(\cdot):=\psi(\cdot)-\psi(2(\cdot))$. Decompose $T_a$ into frequency localised operators $T_a^j$ for integers $j\ge 0$, where $T_a^j$ is defined by
\begin{align*}
\begin{cases}
T_a^0f(x)=\int a(x, \xi)\psi(\xi)\widehat{f}(\xi)e^{ix\cdot\xi}d\xi,\\
T_a^{j}f(x)=\int a(x, \xi)\wt{\psi}(2^{-j}\xi)\widehat{f}(\xi)e^{ix\cdot\xi}d\xi, & j>0.
\end{cases}
\end{align*}
It is well known that the operator $T_a^j$ is naturally associated to the spatial scale $2^{-j\rho}$; this will become clear in the forthcoming Lemma \ref{ttlemma} via an integration by parts argument. 
Thus we further decompose each $T_a^j$ into spatially localised operators $T_a^{j, l}$ defined by
\begin{align*}
T_a^{j, l}f(x)=\int a(x, \xi)\wt{\psi}(2^{-j}\xi)\int f(y)\wt{\psi}(2^{-l+j\rho}(x-y))e^{i(x-y)\cdot\xi}\,dy\,d\xi
\end{align*}
for $l \in \Z$ and $j>0$, and in an analogous manner for $T^{0,l}_a$, so that
\begin{align*}
T_a=\sum_{j\ge 0}T_a^j=\sum_{l\in\mathbb{Z}}\sum_{j\ge 0}T_a^{j, l}.
\end{align*}

We further group the pieces $T_{a}^{j,l}$ according to their spatial scale. Fix some $\epsilon>0$ and write
\begin{align*}
T_a^j=\sum_{l\le j\epsilon}T_a^{j, l}+\sum_{l>j\epsilon}T_a^{j, l}.
\end{align*}
Our $L^p$ improving estimates for the pieces $\sum_{l \leq j \epsilon}$ $T_{a}^{j,l}$ and $T_a^{j,l}$ for $l > j\epsilon$ in what follows, and consequently our sparse bounds for $T_a$, will depend on this fixed value of $\epsilon$. This is the reason why we achieve sparse bounds only up to an endpoint (see Figure \ref{sparse region}).

\subsection*{$L^2\to L^2$ bounds} It is a classical result of Calderón and Vaillancourt \cite{CV1972} that the symbol classes $S^0_{\rho,\rho}$ are bounded on $L^2$ for $0 \leq \rho < 1$. This and an integration by parts argument yield the following $L^2$ bounds.
\begin{lemma}\label{ttlemma}
We have
\begin{align*}
\begin{cases}
\|\sum_{l\le j\epsilon}T_a^{j, l}\|_{L^2\to L^2}\lesssim_{\epsilon} 2^{jm},\\ 
\|T_a^{j, l}\|_{L^2\to L^2}\lesssim_{\epsilon}2^{10n(m-n)(j+l)},& l>j\epsilon.
\end{cases}
\end{align*}
\end{lemma}

\begin{proof}
For $l\ge j\epsilon$, we integrate by parts in the $\xi$ variable $N(\epsilon)$ many times to obtain
\begin{align*}
|T_a^{j, l}f(x)| & \lesssim\int\int|\nabla_{\xi}^{N(\epsilon)}[a(x, \xi)\widetilde{\psi}(2^{-j}\xi)]||x-y|^{-N(\epsilon)}|f(y)|\widetilde{\psi}(2^{-l+j\rho}(x-y))\,dy\,d\xi
\\
& \lesssim_{\epsilon}\|2^{-lN(\epsilon)/2}(1+|\cdot|)^m\widetilde{\psi}(2^{-j}\cdot)\|_{L^1(d\xi)}\int|f(y)|2^{-lN(\epsilon)/2}\widetilde{\psi}(2^{-l+j\rho}(x-y))\,dy
\\
& \lesssim \int|f(y)|2^{-lN(\epsilon)/2}\widetilde{\psi}(2^{-l+j\rho}(x-y))\,dy,
\end{align*}
where $N(\epsilon)$ is taken sufficiently large to guarantee that $2^{-lN(\epsilon)/2}2^{jm}\widetilde{\psi}(2^{-j}(\cdot)) \in L^1$ uniformly in $j$ and $l > j \epsilon$. Imposing also that $N(\epsilon)$ is such that
$2^{-lN(\epsilon)/4}\widetilde{\psi}(2^{-l+j\rho}(\cdot))\in L^1$, uniformly in $j$ and $l > j \epsilon$, and $2^{-lN(\epsilon)/4}\le 2^{10n(m-n)(j+l)}$, an application of Young's inequality yields
\begin{align*}
\|{T_a^{j, l}f}\|_2\lesssim_{\epsilon} 2^{10n(m-n)(j+l)}\|{f}\|_2, \qquad l>j\epsilon,
\end{align*}
as desired.
\newline
\indent 
To prove the first statement, note that for any fixed $j$, we have that $T_a^j=\sum_{l}T_a^{j, l}$ is a pseudodifferential operator associated to a symbol in the class $S_{\rho, \delta}^0$ with constant $2^{jm}$, so $\|\sum_lT_a^{j, l}\|_{L^2\to L^2}\lesssim 2^{jm}$ by the $L^2$ boundedness of the symbol classes. The result then follows from noting that $\sum_{l\le j\epsilon}T_a^{j, l}=\sum_lT_a^{j, l}-\sum_{l>j\epsilon}T_a^{j, l}$, as we just saw that the $L^2$ norm of $T_a^{j,l}$ for $l > j \epsilon$ decays geometrically in $l$.
\end{proof}

\subsection*{$L^1\to L^{\infty}$ bounds}

As in the proof of the $L^2\to L^2$ bounds, integrating by parts in the $\xi$ variable $N(\epsilon)$ many times for a sufficiently large $N(\epsilon)$ in the regime $l>j\epsilon$ leads to the estimate
\begin{align*}
\|T_a^{j, l}f\|_{\infty}\lesssim_{\epsilon} \sup_{x }\int|f(y)|2^{-lN(\epsilon)/2}\widetilde{\psi}(2^{-l+j\rho}(x-y))\,dy\lesssim_{\epsilon}2^{10n(m-n)(j+l)}\|{f}\|_1.
\end{align*}
For $l\le j\epsilon$, we may trivially bound
\begin{align*}
\Big\|\sum_{l\le j\epsilon}T_a^{j, l}f\Big\|_{\infty}\lesssim\int(1+ |\xi|)^m\widetilde{\psi}(2^{-j}\xi)\int|f(y)|\,dy\,d\xi
\lesssim 2^{jm+jn}\|f\|_1.
\end{align*}
We summarize this as the following lemma.
\begin{lemma}\label{oilemma}
We have
\begin{align*}
\begin{cases}
\|\sum_{l\le j\epsilon}T_a^{j, l}\|_{L^1\to L^{\infty}}\lesssim_{\epsilon} 2^{jm+jn},\\ 
\|T_a^{j, l}\|_{L^1\to L^{\infty}}\lesssim_{\epsilon}2^{10n(m-n)(j+l)},& l>j\epsilon.
\end{cases}
\end{align*}
\end{lemma}

\subsection*{$L^{1}\to L^{1}$ bounds}
As above, integrating by parts in the $\xi$ variable $N(\epsilon)$ many times for a sufficiently large $N(\epsilon)$ in the regime $l>j\epsilon$ leads to the estimate
\begin{align*}
\|T_a^{j, l}f\|_{1}\lesssim_{\epsilon} \int \int|f(y)|2^{-lN(\epsilon)/2}\widetilde{\psi}(2^{-l+j\rho}(x-y))\,dy dx\lesssim_{\epsilon} 2^{10n(m-n)(j+l)}\|{f}\|_{1}.
\end{align*}
Reasoning as in Lemma \ref{ttlemma}, the $L^1$ boundedness for $\sum_{l \leq j\epsilon} T_a^{j,l}$ will follow then from that of $T_a^j$. 
\begin{lemma}\label{iilemma}
For any $\varepsilon>0$, we have
\begin{align*}
\begin{cases}
\|\sum_{l\le j\epsilon}T_a^{j, l}\|_{L^{1}\to L^{1}}\lesssim_{\epsilon} 2^{j(m+n(1-\rho)/2 + \varepsilon)},\\ 
\|T_a^{j, l}\|_{L^{1}\to L^{1}}\lesssim_{\epsilon}2^{10n(m-n)(j+l)},& l>j\epsilon.
\end{cases}
\end{align*}
\end{lemma}

\begin{remark}
It is possible to prove the above Lemma with $\varepsilon=0$ using the result of Fefferman \cite{Fe73} that asserts that $T_a:H^1 \to L^1$ for $a \in S^{-n(1-\rho)/2}_{\rho,\delta}$. However, for completeness of the paper, we prefer to give the proof for the above easier lemma than appealing to the result in \cite{Fe73}. The presence of $\varepsilon$ will play no role in determining the (open) trapezoid $\mathcal{T}_{m,\rho,n}$ for which $(r,s')$ sparse bounds hold, provided $\varepsilon>0$ is taken sufficiently small. 
\end{remark}

\begin{proof}
It suffices to show that
\begin{align}\label{eq:L1goal}
\|T_a^j\|_{L^{1}\to L^1}\lesssim 2^{j(m+n(1-\rho)/2+\varepsilon)}
\end{align}
for any $\varepsilon>0$. This in turn follows from
\begin{equation}\label{eq:Linftygoal}
\| T_a^j \|_{L^\infty \to L^\infty} \lesssim 2^{j(m+n(1-\rho)/2)},
\end{equation}
as then one can sum geometrically in $j$ and obtain $L^\infty$ boundedness for the symbol classes $S^m_{\rho,\delta}$ with $m < -n(1-\rho)/2$. By duality, such symbol classes are also bounded on $L^1$ and one readily deduces \eqref{eq:L1goal}.

The inequality \eqref{eq:Linftygoal} may be found in \cite{Fe73}; we include the simple argument for interest of the reader. Write $T_{a}^{j}$ as
$$
T_{a}^{j}f(x)=\int_{\R^n} K_j(x,x-y) f(y)dy,
$$
where
$$
K_j(x,z):=\int_{\R^n} e^{i z \cdot \xi} a(x,\xi) \widetilde{\psi} (2^{-j}\xi) d\xi.
$$
$T_{a}^j$ may be interpreted as the convolution of the function $K(x, \cdot)$ with $f$ evaluated at the point $x$, so it suffices to compute $\|K(x,\cdot)\|_1$ uniformly in $x$.
\newline
\indent
To this end, for some $t>0$ to be determined later, split the integral in two parts,
\begin{align*}
\int_{\R^n} |K_j(x,z)|dz = \int_{|z|<t} |K_j(x,z)|dz + \int_{|z|>t} |K_j(x,z)|dz.
\end{align*}
For the first term, the Cauchy--Schwarz inequality, Plancherel's theorem, and the differential inequalities \eqref{symbol} with $\gamma=0$ yield
\begin{align*}
\int_{|z|<t} |K_j(x,z)|dz & \leq t^{n/2} \Big( \int_{\R^n} |K_j(x,z)|^2dz\Big)^{1/2} \\
& = t^{n/2}  \int_{\R^n} |a(x,\xi)|^2 |\widetilde{\psi}(2^{-j}\xi)|^2d\xi\Big)^{1/2}  \\
& \lesssim t^{n/2} 2^{jn/2} 2^{jm}.
\end{align*}
For the second term we proceed similarly, but exploiting the differential inequalities \eqref{symbol} for $\gamma \neq 0$. Thus,
\begin{align*}
\int_{|z|>t} |K_j(x,z)|dz & \leq \Big(\int_{|z|>t} \frac{1}{|z|^{2N}} \Big)^{1/2} \Big( \int_{\R^n}|z|^{2N} |K_j(x,z)|^2dz\Big)^{1/2} \\
 &\lesssim t^{n/2-N} \Big( \int_{\R^n}\sum_{\gamma \in \N^d: |\gamma|=2N} |D^\gamma [a(x,\xi)\widetilde{\psi}(2^{-j}\xi)]|^2d\xi\Big)^{1/2} \\
& \lesssim t^{n/2-N} 2^{jn/2} 2^{jm - j\rho N},
\end{align*}
provided $N > n/2$. Choosing $t=2^{-j\rho}$, one concludes that
$$
\int_{\R^n} |K_j(x,z)|dz \lesssim 2^{jm+jn(1-\rho)/2}.
$$
\end{proof}


\subsection*{Obtaining $L^r\to L^s$ bounds for $1 \le r\le s \leq 2$}
Interpolating Lemma \ref{ttlemma} and Lemma \ref{iilemma} leads to the bounds
\begin{align*}
\begin{cases}
\| \sum_{l \leq j \epsilon} T_a^{j,l} \|_{L^s \to L^s} \lesssim_{\epsilon} 2^{jm + j(n(1-\rho)/2 + \varepsilon)(\frac{2}{s}-1)} 
\\ 
\|T_a^{j, l}\|_{L^s\to L^s}\lesssim_{\epsilon}2^{10n(m-n)(j+l)},& l>j\epsilon
\end{cases}
\end{align*}
for $1 \leq s \leq 2$.

Using Bernstein's inequality, we have
\begin{align*}
\|T_a^{j, l}f\|_{L^s} & \lesssim_{\epsilon}\|T_a^{j, l}\|_{L^s\to L^s}\|f\ast\mathcal{F}[(\psi(2^{-j-100}\cdot)-\psi(2^{-j+100}\cdot))]\|_{L^s}
\\
& \lesssim\|T_a^{j, l}\|_{L^s\to L^s}2^{-jn(\frac{1}{s}-\frac{1}{r})}\|f\|_{L^r}.
\end{align*}
Similarly,
\begin{align*}
\Big\| \sum_{l\le j\epsilon}T_a^{j, l}f \Big\|_{L^s}& \lesssim_{\epsilon}  \Big\|\sum_{l\le j\epsilon}T_a^{j, l} \Big\| _{L^s\to L^s} \| f\ast\mathcal{F}[(\psi(2^{-j-100}\cdot)-\psi(2^{-j+100}\cdot))]\|_{L^s}
\\
& \lesssim\Big\| \sum_{l\le j\epsilon}T_a^{j, l}\Big\|_{L^s\to L^s}2^{-jn(\frac{1}{s}-\frac{1}{r})}\Norm{f}_{L^r}.
\end{align*}
This leads to the following lemma.

\begin{lemma}\label{qrlemma}
For $1\le r\le s \le 2$, we have
\begin{align*}
\begin{cases}
\|\sum_{l\le j\epsilon}T_a^{j, l}\|_{L^r\to L^s}\lesssim_{\epsilon}2^{jm}2^{-jn(\frac{1}{s}-\frac{1}{r})}2^{j(n(1-\rho)/2 + \varepsilon)(\frac{2}{s}-1)}
\\ 
\|T_a^{j, l}\|_{L^r\to L^s}\lesssim_{\epsilon}2^{-jn(\frac{1}{s}-\frac{1}{r})}2^{10n(m-n)(j+l)},& l>j\epsilon.
\end{cases}
\end{align*}
\end{lemma}

\subsection*{Obtaining $L^r\to L^s$ bounds for $1 \le r\le 2 \leq s \leq r' $}

Interpolating Lemma \ref{ttlemma} and Lemma \ref{oilemma} leads to the bounds
\begin{align}\label{rrprime}
\begin{cases}
\| \sum_{l \leq j \epsilon} T_a^{j,l} \|_{L^r \to L^{r'}} \lesssim_{\epsilon} 2^{jm + jn(\frac{2}{r}-1)}
\\ 
\|T_a^{j, l}\|_{L^r\to L^{r'}}\lesssim_{\epsilon}2^{10n(m-n)(j+l)},& l>j\epsilon
\end{cases}
\end{align}
for $1 \leq r \leq 2$. Reasoning as above, Lemma \ref{ttlemma} and Bernstein's inequality leads to the bounds
\begin{align}\label{2r}
\begin{cases}
\| \sum_{l \leq j \epsilon} T_a^{j,l} \|_{L^r \to L^{2}} \lesssim_{\epsilon} 2^{jm -  jn(\frac{1}{2} - \frac{1}{r})}
\\ 
\|T_a^{j, l}\|_{L^r\to L^{2}}\lesssim_{\epsilon}2^{10n(m-n)(j+l)} 2^{-jn(\frac{1}{2}-\frac{1}{r})} ,& l>j\epsilon.
\end{cases}
\end{align}
Interpolating \eqref{rrprime} and \eqref{2r} leads to the following estimate.

\begin{lemma}\label{r2slemma}
For $1\le r\le 2 \le s \le r'$, we have
\begin{align*}\label{r2s}
\begin{cases}
\|\sum_{l\le j\epsilon}T_a^{j, l}\|_{L^r\to L^s}\lesssim_{\epsilon}2^{jm - jn(\frac{1}{s} - \frac{1}{r})}
\\ 
\|T_a^{j, l}\|_{L^r\to L^s}\lesssim_{\epsilon}2^{-jn(\frac{1}{s}-\frac{1}{r})}2^{10n(m-n)(j+l)},& l>j\epsilon.
\end{cases}
\end{align*}
\end{lemma}

\subsection*{Obtaining geometrically decaying sparse bounds}
The sparse domination for $T_a$ follows from establishing the corresponding sparse domination results for $\sum_{l \leq j \epsilon} T_{a}^{j,l}$ and $\sum_{l > j \epsilon} T_a^{j,l}$. Note that we have much better $L^r \to L^s$ estimates for the terms $T_a^{j, l}$ with $l>j\epsilon$, so these terms will not determine the values of $r$ and $s$ for which we achieve a $(r, s')$ sparse domination result, and we will deal with those tails later.
\newline
\indent
We first invoke the standard ``$3^n$ grid trick". Define the standard universal shifted dyadic grids $\mathcal{D}_{\vec{v}}$ as
\begin{align*}
\bigg\{2^{k}\bigg[m_1+\frac{v_1}{3}, m_1+1+\frac{v_1}{3}\bigg]\times\cdots\times 2^k\bigg[m_n+\frac{v_n}{3}, m_n+1+\frac{v_n}{3}\bigg], k\in\mathbb{Z}, \vec{m}\in\mathbb{Z}^n\bigg\}
\end{align*}
where $\vec{v}\in\{0, 1, 2\}^n$. We can then decompose each operator $T_a^{j, l}$ as
\begin{align*}
T_a^{j, l}=\sum_{\vec{v}}T_{a, \vec{v}}^{j, l}
\end{align*}
where we define $T_{a, \vec{v}}^{j, l}$ by
\begin{align*}
T_{a, \vec{v}}^{j, l}f:=\sum_{\substack{Q\in\mathcal{D}_{\vec{v}}\\l(Q)=2^{\lfloor{-j\rho+j\epsilon +10}\rfloor}}}T_{a}^{j, l}(f\chi_{\frac{1}{3}Q}) \qquad \qquad \text{if} \:\: l \leq j \epsilon,
\end{align*}
and
\begin{align*}
T_{a, \vec{v}}^{j, l}f:=\sum_{\substack{Q\in\mathcal{D}_{\vec{v}}\\l(Q)=2^{\lfloor{-j\rho+l+10}\rfloor}}}T_{a}^{j, l}(f\chi_{\frac{1}{3}Q}) \qquad  \qquad \text{if} \:\: l > j \epsilon.
\end{align*}
We choose to use this decomposition because now the support of each term $T_{a}^{j, l}(f\chi_{\frac{1}{3}Q})$ lies entirely in $Q$. Clearly, it suffices to prove the desired results with the standard dyadic grid replaced with a specific shifted dyadic grid $\mathcal{D}_{\vec{v}}$ for some fixed choice of $\vec{v}\in\{0, 1, 2\}^n$, and with the operator $T_a$ replaced with the operator $T_{a, \vec{v}}$ defined by
\begin{align*}
T_{a, \vec{v}}:=\sum_{j, l}T_{a, \vec{v}}^{j, l}.
\end{align*}
Thus, in what follows, we will suppress the additional subscript $\vec{v}$ in our notation and write $T_a$ and $T_{a}^{j, l}$ in place of $T_a^{\vec{v}}$ and $T_{a, \vec{v}}^{j, l}$ and also implicitly assume that our dyadic cubes lie in the shifted dyadic grid $\mathcal{D}_{\vec{v}}$, for some specific choice of $\vec{v}$.

With this in mind, observe first that the localisation and boundedness of $\sum_{l \leq j \epsilon} T_a^{j,l}$ gives the domination
\begin{align*}
\Big<\sum_{l \leq j \epsilon} T_a^{j,l} f, g\Big> & \lesssim\sum_{j\ge 0}\sum_{\substack{Q\text{ dyadic}: \\ \ell(Q)= 2^{\lfloor{-j\rho+j\epsilon+10}\rfloor}}}\Big\| \sum_{l\le j\epsilon}T_a^{j, l}f\Big\|_{L^s(Q)}\Norm{g}_{L^{s'}(Q)}
\\
 & \lesssim_{\epsilon}\sum_{j\ge 0}\sum_{\substack{Q\text{ dyadic}: \\ \ell(Q)= 2^{\lfloor{-j\rho+j\epsilon+10\rfloor}}}} \!\!\!\!  |Q|^{1/r+1/s'-1}\Big\|\sum_{l\le j\epsilon}T_a^{j, l}\Big\|_{L^r\to L^s}|Q| \left<f\right>_{r, Q}\left<g\right>_{s', Q}.
\end{align*}

To obtain $(r, s')$ sparse bounds for $1\le r\le s \le 2$, we use Lemma \ref{qrlemma},
\begin{align*}
\Big<\sum_{l \leq j \epsilon} T_a^{j,l}f, g\Big> 
\lesssim_{\epsilon}\sum_{j\ge 0}2^{(-j\rho+j\epsilon)n(1/r-1/s)}2^{jm}2^{-jn(\frac{1}{s}-\frac{1}{r})}2^{j(n(1-\rho)/2 + \varepsilon)(2/s-1)}
\\
\times\sum_{\substack{Q\text{ dyadic}: \\ \ell(Q)= 2^{\lfloor{-j\rho+j\epsilon+10}\rfloor}}}|Q|\left<f\right>_{r, Q}\left<g\right>_{s', Q}.
\end{align*}
Note that this gives geometrically decaying $(r, s')$ sparse bounds for sufficiently small choice of $\epsilon>0$ and $\varepsilon>0$ whenever we have
$$
m<-n(1-\rho)(1/r-1/2), \qquad 1\le r\le s \le 2.
$$

To obtain $(r, s')$ sparse bounds for $1\le r\le 2 \le s \le r'$, we use Lemma \ref{r2slemma},
\begin{align*}
\Big<\sum_{l \leq j \epsilon} T_a^{j,l}f, g\Big> 
\lesssim_{\epsilon} & \sum_{j\ge 0}2^{(-j\rho+j\epsilon)n(1/r-1/s)}2^{jm - jn(\frac{1}{s} - \frac{1}{r})}
\\
& \quad \quad \quad  \quad \quad \quad \quad \times\sum_{\substack{Q\text{ dyadic}: \\ \ell(Q)= 2^{\lfloor{-j\rho+j\epsilon+10}\rfloor}}}|Q|\left<f\right>_{r, Q}\left<g\right>_{s', Q}.
\end{align*}
Note that this gives geometrically decaying $(r, s')$ sparse bounds for sufficiently small choice of $\epsilon>0$ whenever we have
$$
m<-n(1-\rho)(1/r-1/s), \qquad 1\le r\le 2 \le s \le r'.
$$

By duality and the fact that symbol classes are closed under adjoints, $(r, s')$ sparse bounds also imply $(s, r')$ sparse bounds, completing the proof of Theorem \ref{mainps} at the expense of establishing the corresponding sparse bounds for the term $\sum_{l > j \epsilon} T_a^{j,l}$ (see Figure \ref{f:2}).

\begin{figure}[H]
\begin{tikzpicture}[scale=3] 
\draw[thick,->] (-.2,0) -- (1.2,0) node[below] {$ \frac 1 r$};
\draw[thick,->] (0,-.2) -- (0,1.2) node[left] {$ \frac{1}{s'}$};

\draw (.5,.05) -- (.5,-.05) node[below] {$ \tfrac 12$};
\draw (.05,.5) -- (-.05,.5) node[left] {$ \tfrac 12$};

\draw (1,.05) -- (1,-.05) node[below] {$ 1$};
\draw (.05,1) -- (-.05,1) node[left] {$ 1$};


\draw[thick] (0,1) -- (1.,1.)  -- (1.,0); 
\draw[thick] (0,1) -- (0.5,0.5)  -- (1.,0); 
\draw[thick]  (0.5,0.5)  -- (1.,1);

\draw[thick]  (0.5,0.5)  -- (1.,.5);
\draw[thick]  (0.5,0.5)  -- (.5,1);

\node at (3/8,7/8) {$A'$};
\node at (5.4/8,7/8) {$B'$};
\node at (7/8,5.4/8) {$B$};
\node at (7/8,3/8) {$A$};

\end{tikzpicture}

\caption{The sparse bounds in the region $A$ corresponds to those under the constraint $1 \leq r \leq s \leq 2$. The ones in the region $B$ are those for which $1 \leq r \leq 2 \leq s \leq r' \leq \infty$. Those in $A'$ and $B'$ are obtained by duality.} 
\label{f:2}
\end{figure}
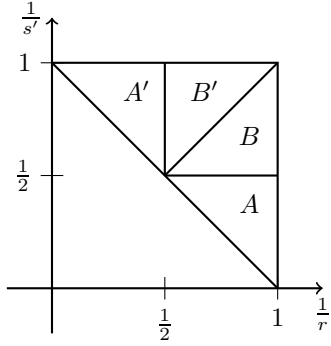

\subsection*{Dealing with the tails}

We only concern ourselves with the case $1 \leq r \leq s \leq 2$; the case $1 \leq r \leq 2 \leq s \leq r'$ follows in an analogous manner. 

For the terms $T_a^{j,l}$, we have
\begin{align*}
\Big< \sum_{j \geq 0} \sum_{l > j\epsilon} T_a^{j,l}f, g \Big> & \lesssim \sum_{j \geq 0} \sum_{l > j\epsilon} \sum_{\substack{Q \text{ dyadic}: \\ \ell(Q)=2^{\lfloor -j\rho+l+10 \rfloor}}} \|T_a^{j,l} f\|_{L^s(Q)} \|g\|_{L^{s'}(Q)} \\ 
& \lesssim_{\epsilon} \sum_{j \geq 0} \sum_{l > j\epsilon} \sum_{\substack{Q \text{ dyadic}: \\ \ell(Q)=2^{\lfloor -j\rho+l+10 \rfloor}}} \!\!\!\! |Q|^{1/r+1/s'-1} \|T_a^{j,l}\|_{L^r \to L^s} |Q| \left< f \right>_{r,Q} \left< g \right>_{s',Q} \\
& \lesssim_{\epsilon}  \sum_{j \geq 0} 2^{(-j\rho)n(1/r-1/s)  -jn(1/s-1/r) +j10n(m-n)} \sum_{l > j\epsilon} 2^{ln(1/r-1/s) +10n(m-n)l} \\
 & \qquad \qquad \qquad \qquad \times \sum_{\substack{Q \text{ dyadic}: \\ \ell(Q)=2^{\lfloor -j\rho+l+10 \rfloor}}} |Q| \left< f \right>_{r,Q} \left< g \right>_{s',Q}
\end{align*}
The sum in $j$ is finite if
$$
m< n - (1-\rho)(1/r-1/s)/10,
$$
which trivially holds as $\rho \in [0,1]$, $0 \leq 1/r-1/s \leq 1$ and $m<0$. Similarly, the sum in $l$ is finite if
$$
10(m-n)<-(1/r-1/s),
$$
which holds from the same considerations as above.

\subsection*{Obtaining true sparse bounds} Observe that for each $j$ and $l$, the partition of $\R^n$ into dyadic cubes $Q$ such that $\ell(Q)=2^{\lfloor{-j\rho+l+10}\rfloor}$ is a sparse family of dyadic cubes, by disjointness of the cubes. The geometric decay in both $j$ and $l$ allows then to recover a true sparse bound, from the following fact about sparse bounds, which tells us that there exists a universal sparse form that controls them all. This type of result has appeared several times in the literature, and we refer the reader to, among others, the work of Lacey and Mena in \cite{LM}, Lemma $4.7$ and the work of Conde-Alonso and Rey in \cite{CAR2014}, Proposition $2.1$. 

\begin{lemma} \normalfont{(}\cite[Lemma 4.7]{LM}\normalfont{).}
Given $f, g$, there is a sparse family of dyadic cubes $\CS^0$ and a constant $C>1$ so that for any other bilinear sparse form $\Lambda_{r,s,\CS}$, we have
$$
\Lambda_{r,s,\CS}(f,g) \leq C \Lambda_{r,s,\CS^0}(f,g).
$$
\end{lemma}
For completeness, we find instructive to provide the reader with all the details, as in the case one has a geometrically decaying series of sparse bounds the obtention of the genuine sparse bounds may be done through the following simple selection process. Let $C>1$ denote the geometrically decaying constant in our sparse bounds, 
so that we have
\begin{equation}\label{firstgeneral}
\left<T_af, g\right>\lesssim\sum_{j\ge 0}C^{-j}\sum_{Q\text{ dyadic}:\,\ell(Q)= 2^{\lfloor{-j\rho+j\epsilon+10}\rfloor}}|Q|\left<f\right>_{r, Q}\left<g\right>_{s', Q}.
\end{equation}
Since a finite union of sparse collections is sparse, it suffices to prove sparse domination for
\begin{align}\label{suff}
\sum_{j\ge 0:\, j\equiv a(\text{mod }N)}C^{-j}\sum_{Q\text{ dyadic}:\,\ell(Q)= 2^{\lfloor{-j\rho+j\epsilon+10}\rfloor}}|Q|\left<f\right>_{r, Q}\left<g\right>_{s', Q}
\end{align}
for some $N$, and for all $0\le a\le N-1$. Choose $N$ sufficiently large so that $C^{N}>4^n$ and $N>100/\rho$ if $\rho\ne 0$, so that different $j$ in \eqref{suff} correspond to different dyadic levels of cubes; this is in contrast with the situation in \eqref{firstgeneral}.  Fix some such $a$. 
After a re-indexing, this large choice of $N$ allows us to rewrite (\ref{suff}) as
\begin{align*}
\sum_{j\ge 0}3^{-nj}\sum_{Q\text{ dyadic}:\,\ell(Q)=2^{-m(j)+j_0}}|Q|\left<f\right>_{r, Q}\left<g\right>_{s', Q},
\end{align*}
where $m(j)$ is a strictly increasing sequence of integers such that $m(0)=0$.
Define our sparse collection $\mathcal{S}$ as follows. For each $j$, choose $\mathcal{S}_j$ to be a certain \textit{checkerboard} collection of dyadic cubes of sidelength $2^{-m(j)+j_0}$: for each dyadic cube $Q$ of sidelength $2^{-m(j)+j_0+j}$, choose a single dyadic cube $Q'$ of sidelength $2^{-m(j)+j_0}$ contained in $Q$ such that $\left<f\right>_{r, Q'}\left<g\right>_{s', Q'}$ is maximized. Then if we set $\mathcal{S}=\bigsqcup_j\mathcal{S}_j$, clearly
\begin{multline*}
\sum_{j\ge 0}3^{-nj}\sum_{Q\text{ dyadic}:\,\ell(Q)=2^{-m(j)+j_0}}|Q|\left<f\right>_{r, Q}\left<g\right>_{s', Q}
\\
\lesssim\sum_{j\ge 0}\sum_{Q\in \mathcal{S}_j}|Q|\left<f\right>_{r, Q}\left<g\right>_{s', Q}\lesssim\sum_{Q\in\mathcal{S}}|Q|\left<f\right>_{r, Q}\left<g\right>_{s', Q}.
\end{multline*}
Note that $\mathcal{S}$ is clearly a Carleson collection of cubes, i.e. it is a collection of dyadic cubes such that the total measure of cubes of a given sidelength decays geometrically with the sidelength,
$$
\sum_{Q \in \CS} |Q| = \sum_{j \geq 0} \sum_{Q \in \mathcal{S}_j} |Q| = \sum_{j \geq 0} |Q_0| 2^{n [m(j) - j_0 - j]} 2^{-n[m(j) - j_0]} = \sum_{j \geq 0} 2^{-jn} |Q_0| \leq |Q_0|,
$$
where $Q_0$ is the cube containing the supports of $f$ and $g$. Therefore $\CS$ is sparse, as it is well known that the Carleson and sparse conditions are equivalent; see \cite{LN}.

\section{Sharpness of the sparse bounds}\label{sec:endpoints}
Suppose towards a contradiction that a sparse bound holds outside the closure of the range given by Theorem \ref{mainps}. That is, suppose that for any symbol $a\in S^{m}_{\rho, \delta}$ there are some $r$ and $s$ such that 
\begin{align}\label{hyp}
m>-n(1-\rho)(1/r-1/2),\qquad 1\le r\le s\le 2
\end{align}
or
\begin{align*}
m>-n(1-\rho)(1/r-1/s),\qquad 1\le r\le 2\le s\le r'
\end{align*}
and for which a sparse $(r, s')$ bound holds; the case for which a sparse $(s', r)$ bound holds follows from duality. In particular, we may take this symbol $a\in S^m_{\rho, \delta}$ to be $a(x,\xi)=e^{i|\xi|^{1-\rho}}(1+|\xi|)^m$ if  $1 \leq r \leq  s\le 2$ and $a(x,\xi)=(1+|\xi|^2)^{m/2}$ if $r\le 2\le s$, since these symbols satisfy the same $L^p$ improving boundedness properties as a general symbol for these ranges of $r$ and $s$.
\newline
\indent
Consider first the case of $1 \leq r\le s\le 2$ satisfying \eqref{hyp}. Set $m'=-j\rho n(1/r-1/s)$ and define the symbol $b(x, \xi)=a(x, \xi)(1+|\xi|^2)^{m'/2}$. Then $T_b$ is unbounded from $L^p$ to $L^q$ for all $(p, q)$ in a sufficiently small neighbourhood of $(r, s)$, so for sufficiently small $\epsilon'>0$, there are infinitely many $j$ such that  
$$\|T_a^j\|_{L^r\to L^s}\ge 2^{j[\rho n(\frac{1}{r}-\frac{1}{s})+\epsilon']}.$$ 
Let $\{j_k\}_{k=1}^{\infty}$ be a sequence of such $j$'s. But by Lemma \ref{qrlemma}, 
$$\Big\|\sum_{l>j\epsilon}T_a^{j, l}\Big\|_{L^r\to L^s}\lesssim 2^{j\rho n(\frac{1}{r}-\frac{1}{s})},$$
so reindexing the $j_k's$ to start at a sufficiently large value of $j$, we have for all $k$,
\begin{align*}
\Big\|\sum_{l\le j_k\epsilon}T_a^{j_k, l}\Big\|_{L^r\to L^s}\gtrsim 2^{j_k[\rho n(\frac{1}{r}-\frac{1}{s})+\epsilon']}.
\end{align*}
Recall that for $1 \leq r \leq  s\le 2$, we considered $a(x, \xi)=e^{i|\xi|^{1-\rho}}(1+|\xi|)^m$. Integrating by parts in the $\xi$ variable $N(\epsilon)$ times (in the ``reverse" manner to Lemma \ref{ttlemma}, that is applying $\xi$ derivatives to the exponential term $e^{i(x-y)\cdot\xi}$ rather than the symbol), we have
\begin{align*}
|T_a^{j, l}f(x)| & =\Big|\int e^{i|\xi|^{1-\rho}}(1+|\xi|)^m\widetilde{\psi}(2^{-j}\xi)\int f(y)[\wt{\psi}(2^{-l+j\rho}(x-y))]e^{i(x-y)\cdot\xi}\,dy\,d\xi\Big|
\notag \\
& \lesssim_{\epsilon} 2^{jn}(2^{j\rho}2^{l-j\rho})^{N(\epsilon)}\int|f(y)|\wt{\psi}(2^{-l+j\rho}(x-y))|\,dy.
\end{align*}
For $N(\epsilon)$ sufficiently large, 
$$\sum_{l<-j_k\epsilon}\big\|2^{jn}(2^{j\rho}2^{l-j\rho})^{N(\epsilon)}\wt{\psi}(2^{-l+j\rho}\cdot)\big\|_{L^{(1+1/s-1/r)^{-1}}}\lesssim 1,$$
so by Young's convolution inequality and the triangle inequality we have
\begin{align*}
\Big\|\sum_{-j_k\epsilon\le l\le j_k\epsilon}T_a^{j_k, l}\Big\|_{L^r\to L^s}\gtrsim  2^{j_k[\rho n(\frac{1}{r}-\frac{1}{s})+\epsilon']}.
\end{align*}
By pigeonholing, it follows that there is some $l=l(k)$ with $-j_k\epsilon\le l\le j_k\epsilon$ such that
\begin{align*}
\|T_a^{j_k, l(k)}\|_{L^r\to L^s}\gtrsim  2^{j_k[\rho n(\frac{1}{r}-\frac{1}{s})+\epsilon'/2]}.
\end{align*}
Moreover, by disjointness and translation invariance there is a $f_k$, with $\Norm{f_k}_{L^r}=1$, supported in the ball centered at the origin of radius $2^{-j_k\rho+l(k)-10n}$ and such that
\begin{align*}
\|T_a^{j_k, l(k)}f_k\|_{L^s}\gtrsim  2^{j_k[\rho n(\frac{1}{r}-\frac{1}{s})+\epsilon'/2]}.
\end{align*}
By the spatial localisation of the operator $T_a^{j_k, l(k)}$ and the support of $f_k$, the function $T_a^{j_k, l(k)}f_k$ is supported in a dyadic annulus at scale $2^{-j_k \rho + l(k)}$. Then, there exists a $g_k$ supported in that annulus, with $\| g_k \|_{L^{s'}}=1$ and such that
\begin{align}\label{innerproduct}
\Big|\Big<T_a^{j_k, l(k)}f_k, g_k\Big>\Big|\gtrsim 2^{j_k[\rho n(\frac{1}{r}-\frac{1}{s})+\epsilon'/2]}.
\end{align}
Observe that the supports of $f_k$ and $g_k$ are disjoint, and in particular
\begin{align}\label{dist}
\text{dist}(\text{supp}(f_k), \text{supp}(g_k))\approx 2^{-j_k\rho+l(k)}.
\end{align}
Now, for each $k$, $\sum_{l}T_{a}^{j_k, l}$ is a symbol in $S^m_{\rho, \delta}$, so by hypothesis there is a positive sparse form $\Lambda_{\mathcal{S}, r, s', k}(f, g)=\sum_{Q\in\mathcal{S}}|Q|\left<f\right>_r\left<g\right>_{s'}$ such that
\begin{align*}
\Lambda_{\mathcal{S}, r, s', k}(f_k, g_k)\gtrsim \Big|\Big<\sum_{l}T_{a}^{j_k, l}f_k, g_k\Big>\Big|.
\end{align*}
Then, by \eqref{innerproduct} and support considerations
\begin{align*}
\Lambda_{\mathcal{S}, r, s', k}(f_k, g_k) & \gtrsim \Big|\Big<\sum_{l}T_{a}^{j_k, l}f_k, g_k\Big>\Big|
\\
& =\Big|\Big<T_a^{j_k, l(k)}f_k, g_k\Big>\Big|\gtrsim  2^{j_k[\rho n(\frac{1}{r}-\frac{1}{s})+\epsilon'/2]}.
\end{align*}
On the other hand, a cube contributing effectively to the sparse form $\Lambda_{\mathcal{S}, r, s', k}$ must contain the supports of both $f_k$ and $g_k$. By \eqref{dist}, the largest possible value for the sparse form will be given by a sparse family consisting of a single cube $Q$ with smallest possible sidelength satisfying such a property; note that for $r \leq s$ the factor $|Q|^{1-1/r-1/s'}$ decreases with the sidelength of $Q$. Then 
\begin{align*}
\Lambda_{\mathcal{S}, r, s', k}(f_k, g_k)\lesssim 2^{-j_k(\rho+\epsilon)n(1/s-1/r)},
\end{align*}
and so for an appropriate choice of $\epsilon, \epsilon'$, we have
\begin{align*}
-n\rho(1/s-1/r)< -n\rho(1/s-1/r),
\end{align*}
a contradiction.
\newline
\indent
To handle the case when $1\leq r\le 2\le s \leq r'$, the argument is similar. The only difference is that since the $L^p$-improving estimates are $\rho$-independent, we treat the symbol $(1+|\xi|^2)^{m/2}$ as belonging to the best possible symbol class that is a member of, namely $S^m_{1, \delta}$, and physically localised to scale $2^{-j}$. We leave the verification of the details to the interested reader.
\subsection*{Sharpness along the line $v_1-v_4$}

We remark that a quick inspection allows one to identify certain pairs $(r,s')$ for which sparse bounds cannot hold, as otherwise they would imply consequences which are known to be false. In particular, if $-n(1-\rho)/2 < m <0$, the pairs $(r,s')$ for which $(1/r, 1/s')$ lie outside the equilateral triangle formed by $v_1$, $v_4$ and $v_5$ in the sparse diagram (see Figure \ref{sparse region}) are not admissible; here $v_5:=(\frac{1}{2} + \frac{-m}{n(1-\rho)}, \frac{1}{2} + \frac{-m}{n(1-\rho)})$. Indeed, by duality, it is enough to show that no sparse bounds may hold if $(1/r,1/s')$ lie inside the triangle joining $(1/2,1/2), (1,1)$ and $(1,0)$ but outside the triangle joining $(1/2,1/2)$, $v_1$ and $v_5$.

Suppose towards a contradiction that given any $a \in S^m_{\rho,\delta}$ there exists a $(r,s')$ sparse bound for $T_a$ with
\begin{equation}\label{value of r contradiction}
\frac{1}{r} > \frac{1}{2} + \frac{-m}{n(1-\rho)}
\end{equation}
and some $s\geq r$. By Lemma \ref{WeightedSparse}, this implies that $T_a$ is bounded on $L^p$ for $r < p < s$.  Rearranging \eqref{value of r contradiction} one has
$$
r < \frac{2n(1-\rho)}{n(1-\rho)-2m},
$$
so there exists $\varepsilon>0$ such that $T_a$ is bounded on $L^{p(\varepsilon)}$, with $p(\varepsilon)=\frac{2n(1-\rho)}{n(1-\rho)-2m} - \varepsilon$. But this is a contradiction, as it is shown in \cite{Fe73} that there are symbols in $S^m_{\rho,\delta}$ which fail to be bounded on $L^p$ for $p<\frac{2n(1-\rho)}{n(1-\rho)-2m}$.

\section{Consequences of sparse domination}\label{sec:weighted}

One of the main advantages of sparse domination is that it yields weighted inequalities as a corollary. To obtain weighted estimates for sparse forms has proven to be much simpler than for many other operators, thanks to the fact that they are positive, localised and well behaved objects. We discuss some of them in the upcoming subsections.

\subsection{One-weight estimates}

Let $w$ denote a weight, that is, a nonnegative locally integrable function. Recall that given $1<p<\infty$, the Muckenhoupt classes of weights $A_p$ and the reverse Hölder classes of weights $RH_p$ consist of all $w$ satisfying
$$
[w]_{A_p}:=\sup_{B} \: \langle w \rangle_B \: \langle w^{1-p'} \rangle^{p-1}_B < \infty, \quad \quad [w]_{RH_p}:=\sup_{B} \: \langle w \rangle_B^{-1} \: \langle w \rangle_{B,p} < \infty
$$
respectively, where the supremum ranges over all balls $B$ in $\R^n$. The class $A_1$ consists of all $w$ satisfying $Mw \leq C w$ for some finite constant $C>0$; here $M$ denotes the Hardy--Littlewood maximal function. The infimum of such admissible constants $C$ is denoted by $[w]_{A_1}$.

Quantitative weighted estimates in the context of the above weights may be deduced for operators dominated by bilinear sparse forms as a consequence of the following lemma.
\begin{lemma}[\cite{BFP}]\label{WeightedSparse}
Suppose that we are given $r, s\in [1, \infty]$ and a sparse collection $\mathcal{S}$ of dyadic cubes. Then for any $r<p<s$ there exists a constant $C=C(r, s, p)$ so that for every weight $w\in A_{p/r}\cap RH_{(s/p)'}$,
\begin{equation*}
\Lambda_{\mathcal{S}, r, s'}(f, g)\le C\big([w]_{A_{p/r}}[w]_{RH_{(s/p)'}}\big)^{\alpha}\|f\|_{L^p(w)}\|g\|_{L^{p'}(w^{1-p'})},
\end{equation*}
where $\alpha:=\max\Big(\frac{1}{p-r}, \frac{s-1}{s-p}\Big).$
\end{lemma}
This lemma is sharp at the level of sparse bounds and provides the best possible dependence on the $A_p$ characteristic of the weight. The intersection of the $A_p$ and the reverse Hölder classes may be understood as
\begin{equation}\label{equivalenceApRH}
w \in A_{p/r} \cap RH_{(s/p)'} \Longleftrightarrow w^{(s/p)'} \in A_{(s/p)'(p/r-1 )+1}.
\end{equation}
Of course Theorem \ref{mainps} and Lemma \ref{WeightedSparse} allows one to deduce weighted estimates for pseudodifferential operators associated to the symbol classes $S^m_{\rho,\delta}$. In views of the reverse Hölder properties of weights, one may effectively use a sparse bound at the closure of the trapezoid $\mathcal{T}_{m,\rho,n}$ in order to deduce weighted estimates. In particular, using the endpoint $v_3$, one has the following.

\begin{corollary}\label{weightedConsequences}
Let $a \in S^m_{\rho, \delta}$ for some $m<0$ and $0<\delta\le\rho<1$.

\begin{enumerate}[ \normalfont (i)]
\item Let $m=-n(1-\rho)$. Then $T_a$ is bounded on $L^p(w)$ for $w \in A_p$ and all $1<p<\infty$. \label{item1}
\item Let $-n(1-\rho)<m\leq -n(1-\rho)/2$. Then $T_a$ is bounded on $L^p(w)$ for $w \in A_{p/r_e}$ and all $r_e<p<\infty$, where $r_e:=-n(1-\rho)/m$. \label{item2}
\item Let $-n(1-\rho)/2 < m < 0$. Then $T_a$ is bounded on $L^p(w)$ for $w \in A_{p/2} \cap RH_{(s_e/p)'}$ and all $2<p<s_e$, where $s_e:=\frac{2n(1-\rho)}{n(1-\rho) + 2m}$. \label{item3}
\end{enumerate}
\end{corollary}

This corollary recovers all the previous known weighted estimates for pseudodifferential operators in the context of Muckenhoupt weights mentioned in the Introduction, except for the cases $p=r_e$ in \eqref{item2} when $m=-n(1-\rho)/2$ and $p=2$ in \eqref{item3}; see \cite{MRScan} for \eqref{item1}, \cite{CT, MRS10} for \eqref{item2} and \cite{Bel2016} for \eqref{item3}. We believe that the estimates \eqref{item2} for $-n(1-\rho)<m<-n(1-\rho)/2$ are new. One should note that the estimates in Corollary \ref{weightedConsequences} only contain some of the possible weighted estimates; in particular other versions for \eqref{item2} and \eqref{item3} may be obtained by using other sparse bounds in the admissible region $\mathcal{T}_{m,\rho,n}$.

We also note that one could state a version of Corollary \ref{weightedConsequences} with an explicit quantitative dependence of the operator norm in terms of the $[w]_{A_p}$ and $[w]_{RH_p}$ constants. Because of the use of the reverse Hölder inequality in the proof of the corollary, the behaviour of the weighted constant would not just be that of Lemma \ref{WeightedSparse}, but would also depend on how the implicit constant in the sparse domination in Theorem \ref{mainps} depends on $r$ and $s$. The resulting quantitative weighted estimate would probably be far from being sharp and we do not concern with such finer points in this paper.

\begin{proof}[Proof of Corollary \ref{weightedConsequences}]
We shall only prove \eqref{item2} and \eqref{item3}, as \eqref{item1} is a particular case of \eqref{item2} when $m=-n(1-\rho)$. Let $a \in S^m_{\rho,\delta}$ with $-n(1-\rho) \leq m \leq -n(1-\rho)/2$, $w \in A_{p/r_e}$ and $r_e < p < \infty$. By the reverse Hölder property of weights, there exists $\varepsilon:=\varepsilon(n,p,[w]_{A_\infty})>0$ such that $w \in A_{p/r_e - \delta}$ for all $0 \leq \delta \leq \varepsilon$; see for instance \cite{HPR} for its sharp version. In particular, choose $\delta < \min \{\varepsilon, (p-r_e)/2r_e\}$ and set $r:=\frac{p}{p/r_e - \delta}$. Then $w \in A_{p/r}$ and $r_e < r < p $. As $(1/r, 1) \in \mathcal{T}_{m,\rho,n}$, $T_a$ is dominated by a $(r,1)$ sparse form by Theorem \ref{mainps}, and therefore bounded on $L^p(w)$ by Lemma \ref{WeightedSparse}. This establishes \eqref{item2}.

Similarly for a given $a \in S^m_{\rho,\delta}$, with $-n(1-\rho)/2<m<0$, let $w \in A_{p/2} \cap RH_{(s_{e}/p)'}$ and $r< p<s_{e}$. By the characterisation \eqref{equivalenceApRH}, $w^{(s_e/p)'} \in A_{(s/p)'(p/2-1)+1}$. Using the reverse Hölder property of weights \cite{HP}, there exists $\varepsilon:=\varepsilon(n,[w]_{A_\infty})>0$ such that
$$
\langle w^{(s_{e}/p)'}  \rangle_{1+\delta,Q} \leq 2 \langle w^{(s_{e}/p)'}  \rangle_{1,Q}
$$
for any cube $Q$ and $0 \leq \delta \leq \varepsilon$. Thus $$
\langle w \rangle_{(s_{e}/p)'(1+\delta),Q} \leq 2 \langle w \rangle_{(s_{e}/p)',Q} \lesssim \langle w \rangle_{1,Q},
$$
where the last inequality follows as $w \in RH_{(s_{e}/p)'}$. In particular, we have that $w \in RH_{(s_{e}/p)'(1+\delta)}$ for $0 \leq \delta \leq \varepsilon$. It is then a computation to find $p<s<s_e$ such that $w \in A_{p/2} \cap RH_{(s/p)'}$, with $(1/2,1/s')$ belonging to $\mathcal{T}_{m,\rho,n}$. It follows then from Theorem \ref{mainps} that $T_a$ is dominated by a $(2,s')$ sparse form, so by Lemma \ref{WeightedSparse} $T_a$ is bounded on $L^p(w)$, which concludes the proof of \eqref{item3}.
\end{proof}

\subsection{Fefferman--Stein inequalities}
It is also possible to deduce Fefferman--Stein weighted inequalities as a consequence of the sparse domination. That is, given $1 \leq p < \infty$, to identify an operator $\CM$ such that
\begin{equation}\label{FS}
\int_{\R^n} |Tf|^p w \leq C \int_{\R^n} |f|^p \CM w
\end{equation}
holds for any weight $w$, with constant independent of the weight. 
\newline
\indent
Lemma \ref{WeightedSparse} immediately yields this type of estimates for operators which are sparse dominated. This follows from the trivial observation that $(M_\gamma w)^{(s/p)'} \in A_1$ if $\gamma > (s/p)'$, with constant independent of $w$; here $M_\gamma w := (Mw^\gamma)^{1/\gamma}$. As $w \leq M_\gamma w$, an operator $T$ dominated by a $(r,s')$ sparse form satisfies then the inequality \eqref{FS} with $\CM=M_\gamma$ for $\gamma > (s/p)'$ and $r <p<s$.

In the case of $(r,1)$-sparse forms, where the condition on $\gamma$ becomes $\gamma>1$, the first author observed in \cite{Bel2015} that $M_\gamma$ may be improved to the controlling maximal function $\CM=M^{\lfloor p \rfloor +1}$; here $M^{\lfloor p \rfloor +1}$ denotes the $(\lfloor p \rfloor +1)$-fold composition of $M$ with itself. In this case, the argument is more intrinsic and does not follow from Lemma \ref{WeightedSparse}. The question whether $M_\gamma$ may be improved for general $(r,s)$ sparse forms remains open.

Of course the above discussion yields Fefferman--Stein inequalities for pseudodifferential operators through the sparse bounds in Theorem \ref{mainps}. We remark that the first author proved in \cite{Bel2016} sharp versions of these inequalities for $p=2$ and for any given symbol class $S^m_{\rho,\delta}$, with the controlling operator $\CM$ being a fractional maximal function associated to non-tangential approach regions. The proof of that result differs very much from the techniques presented in this paper. It would be interesting to explore a further connection between the two; in particular whether if sparse domination techniques could lead to Fefferman--Stein inequalities with operators $\CM$ falling beyond the classical $A_p$ theory, in a way that allows one to recover the results in \cite{Bel2016}.

\subsection{Coifman--Fefferman estimates}

It has recently been established in \cite{LPRR} that if an operator is dominated by a $(1,s)$ sparse form for any $s>1$, then a Coifman--Fefferman type estimate
$$
\|Tf\|_{L^p(w)} \lesssim  \|Mf\|_{L^p(w)}.
$$
holds for any $w \in A_\infty$ and $1 \leq p < \infty$, with implicit constant depending on $[w]_{A_\infty}$. This applies then to pseudodifferential operators associated to the symbol classes $S^{-n(1-\rho)}_{\rho,\delta}$. We believe this is the first time that these sorts of estimates are obtained for pseudodifferential operators.

\subsection{Weak-type estimates}

Theorem \ref{mainps} also allows to recover that the symbol classes $S^{m}_{\rho,\delta}$ with $m < -n(1-\rho)/2$ and $0 \leq \delta \leq \rho < 1$ are of weak-type $(1,1)$. This is a consequence of the following result for sparse forms. 

\begin{lemma}[\cite{CACDO}]
Suppose that a sublinear operator $T$ has the following property: there exists $C > 0$
and $1 \leq  r < \infty$ such that for every $f, g$ bounded with compact support there exists a sparse
collection $\CS$ such that
$$
| \left< T f, g \right>| \leq C \Lambda_{\CS,1,r}(f,g).
$$
Then $T : L^1(\R^n) \to L^{1,\infty}(\R^n)$ boundedly.
\end{lemma}

\section{Oscillatory Fourier multipliers}\label{sec:multipliers}

Theorem \ref{mainps} and the techniques presented in Section \ref{sec:proof} naturally extend for certain classes of oscillatory multipliers. In this section, given $\alpha,\beta \in \R$, let $m$ denote a Fourier multiplier on $\mathbb{R}^n$, with support in the set $\{\xi \in \R^d : |\xi|^\alpha \geq 1\}$, and satisfying the Miyachi condition
\begin{equation}\label{MikhlinCondition}
|D^\gamma m(\xi)|\lesssim |\xi|^{-\beta+|\gamma|(\alpha-1)}
\end{equation}
for every multi-index $\gamma \in \N^n$.
\newline
\indent
Of course, when $0< \alpha \leq 1$, the differential inequalities \eqref{MikhlinCondition} are essentially equivalent to those defining the symbol classes $S^{-\beta}_{1-\alpha,0}$. The reason to have the support condition $\{|\xi|^\alpha \geq 1\}$ and the factor $|\xi|$ in \eqref{MikhlinCondition} instead of $(1+|\xi|)$ is for our treatment to hold also when $\alpha<0$. In that case, the relevant behaviour of the oscillatory multipliers is determined by low frequencies. Interestingly, when $\alpha \not \in [0,1]$, differentiating the multiplier implies a growth in its order; this is in contrast with the case $0 \leq \alpha < 1$, where differentiating implies decay. Model examples for the above Fourier multipliers are given by
$
m_{\alpha,\beta}(\xi):=|\xi|^{-\beta}e^{i|\xi|^\alpha}\chi_{\{|\xi|^\alpha \geq 1\}}(\xi)
$, first studied by Hirschman \cite{Hi59}, and later by Wainger \cite{Wainger}, Fefferman \cite{Fe70}, Fefferman and Stein \cite{FSmult}, Miyachi \cite{Mi80,Mi81} and others. They satisfy the following sparse bounds.

\begin{theorem}\label{MultipliersSparse}
Let $\alpha, \beta \in \R$, $\alpha \neq 0$. Let $m$ be a multiplier satisfying \eqref{MikhlinCondition}. Then for any compactly supported bounded functions $f, g$ on $\mathbb{R}^n$, there exists sparse collections $\mathcal{S}$ and $\widetilde{\mathcal{S}}$ of dyadic cubes such that
\begin{align*}
|\left<Tf, g\right>|\le C(\alpha,\beta, r, s)\Lambda_{\mathcal{S}, r, s'}(f, g)
\end{align*}
and
\begin{align*}
|\left<Tf, g\right>|\le C(\alpha,\beta, r, s)\Lambda_{\widetilde{\mathcal{S}}, s', r}(f, g)
\end{align*}
for all pairs $(r, s')$ and $(s', r)$ such that $\alpha \cdot  \beta>0$ and
\begin{align*}
|\beta|>n|\alpha|(1/r-1/2), \qquad 1\le r\le s \le 2
\end{align*}
or
\begin{align*}
|\beta|>n |\alpha| (1/r-1/s), \qquad 1\le r\le 2 \le s \le r'.
\end{align*}
\end{theorem}
The missing case $\alpha=0$, $\beta\neq 0$ in the above theorem corresponds to the classical case of fractional integrals, for which \textit{fractional} sparse bounds were obtained by Cruz-Uribe and Moen in \cite{CUM}, and the case $\alpha=\beta=0$ corresponds to the classical Mikhlin multipliers; see for instance \cite{BCA} for sparse bounds. To approach the $\alpha =0$ case with our methods, one needs to consider both low and high frequencies. It is clear that the resulting conditions on $\alpha$ and $\beta$ coming from each case would be excluding one another.
\newline
\indent
For the case $\alpha>0$, the proof is identical to that provided in Section \ref{sec:proof}. Indeed, as now we are in a translation-invariant setting, the $L^2$ bound is just the $L^\infty$ norm of the multiplier in contrast with the more involved $L^2$ theory for pseudodifferential operators.
\newline
\indent
For the case $\alpha <0$, one needs to perform a decomposition of the region $\{|\xi|\leq 1\}$ into dyadic annuli of width $2^j$ for $j<0$. As the sum runs over $j<0$, we require the reverse condition on the exponent of $2^j$ in order to sum the geometric series. Also, note that for $\alpha <0$ we require $\beta<0$, so for the summability on $l$ one requires
$$
\|T_a^{j,l}\|_{L^r \to L^s} \lesssim 2^{-10n(\beta-n)(j-l)}
$$
for all $l>|j|\epsilon$.
\newline
\indent
As it is well known, results for the above classes of multipliers imply as corollaries bounds on oscillatory kernels and the solution operator to certain dispersive equations. Of course, Theorem \ref{MultipliersSparse} also yields weighted corollaries for oscillatory Fourier multipliers in the very same fashion as is presented in Section \ref{sec:weighted}; the details are left to the interested reader.

\subsection{Oscillatory kernels}

For $a>0$, $a\neq 1$ and $b \geq n(1-\frac{a}{2})$, let $K_{a,b}:\R^n \to \C$ be given by
$$
K_{a,b}(x)=\frac{e^{i|x|^a}}{(1+|x|)^b}.
$$
These convolution kernels satisfy the sparse bounds in Theorem \ref{MultipliersSparse} with $\alpha=\frac{a}{a-1}$ and $\beta=\frac{na/2-n+b}{a-1}$. To see this, write $K_{a,b}=K_{a,b}^0 + K_{a,b}^\infty$, where $K_{a,b}^0=K_{a,b}\eta$ for some $\eta \in C^\infty_c$. As $|K_{a,b}^0| \lesssim \eta$, it satisfies all possible sparse bounds. For $K_{a,b}^{\infty}$, we use the well known-fact (see, for instance, \cite{Sjo81Lp, bigStein}) that its Fourier transform satisfies the condition \eqref{MikhlinCondition} for $\alpha=\frac{a}{a-1}$ and $\beta=\frac{na/2-n+b}{a-1}$. We further break the multiplier in two parts, $\widehat{K}^\infty_{a,b}=m_0+m_\infty$, so that $m_\infty$ guarantees the support condition $\{|\xi|^{\alpha} > 1\}$.  As $m_0$ is a rapidly decreasing function, so it is its associated kernel, and therefore also satisfies all possible sparse bounds. Finally, for $m_\infty$, apply Theorem \ref{MultipliersSparse}, which determines the sparse bounds for the original convolution kernel $K_{a,b}$.
\newline
\indent
We note that our results exclude the case $a=1$, as the associated Fourier multipliers are not \textit{oscillatory}. Indeed when $a=1$, $b=(n+1)/2 + \delta$, the kernels $K_{a,b}$ are the convolution kernels associated to the Bochner--Riesz multipliers $m_\delta$. Sparse bounds for Bochner--Riesz multipliers have been recently obtained by Lacey, Mena and Reguera in \cite{LMR}; see also the endpoint result of Kesler and Lacey \cite{KLbr} or the previous work of Benea, Bernicot and Luque \cite{BBL}. Besides the weighted $A_p$ estimates obtained as a consequence of the sparse domination, one should note that there are weighted Fefferman--Stein inequalities for Bochner--Riesz multipliers involving Kakeya-type maximal functions; see for instance the earlier work of Carbery \cite{CarWeight} or Carbery and Seeger \cite{CS}.

\subsection{Dispersive equations}

Given $\alpha \in \N$, the solution of the dispersive equation
\begin{equation*}\label{IVP}
\begin{cases}
i \partial_t u +(-\Delta)^{\alpha/2}u=0 \\
u(\cdot,0)=f
\end{cases}
\end{equation*}
is given by
$$
u(x,t)=e^{it(-\Delta)^{\alpha/2}}f(x):=\int_{\R^n} e^{i x\cdot \xi} e^{i t |\xi|^\alpha} \widehat{f}(\xi)d\xi.
$$
As the Fourier multiplier associated to the $x$-variable is $m(\xi)=e^{it|\xi|^\alpha}$, a rescaling argument provides sparse bounds for $u(x,t)$ for a fixed time $t$. For $\alpha \in \N$, and under the hypothesis of Theorem \ref{MultipliersSparse} one has that there exists a sparse family $\CS$ such that
$$
|\left<u(\cdot, t), g\right>|\le C(m,\rho, r, s) \sum_{Q \in \CS} |t^{1/\alpha}Q| \left< (I-t^{2/\alpha}\Delta)^{\beta/2}f \right>_{r,t^{1/\alpha}Q} \left< g \right>_{s,t^{1/\alpha}Q},
$$
where $t^{1/\alpha}Q$ denotes the concentric cube to $Q$ of volume $|t^{1/\alpha}Q|=t^{n/\alpha}|Q|$ and 
$$
(I-t^{2/\alpha} \Delta)^{\beta/2}f(x):=\int_{\R^n} e^{ix \cdot \xi} (1+t^{2/\alpha}|\xi|^2)^{\beta/2}\widehat{f}(\xi)d\xi
$$
defines a $t-$inhomogeneous Sobolev norm.

\subsection{A subdyadic Hörmander condition}

Theorem \ref{MultipliersSparse} also holds for broader classes of multipliers than those given by the Mikhlin-type condition \eqref{MikhlinCondition}. In particular, it is enough to ask that the multiplier $m$, with support in $\{\xi \in \R^n : |\xi|^\alpha \geq 1\}$, satisfies the Hörmander-type condition \footnote{Our results may also be seen to hold under a weaker Hörmander--Sobolev formulation, although we refrain from doing that for simplicity.}
\begin{equation}\label{HormSD}
\sup_{B}\; \dist(B,0)^{\beta+(1-\alpha)|\gamma|} \Big( \frac{1}{|B|}\int_B |D^\gamma m(\xi)|^2 d\xi \Big)^{1/2}< \infty
\end{equation}
for all $\gamma \in \N^n$. Here the supremum is taken over all euclidean balls $B$ in $\mathbb{R}^n$ with $\dist(B,0)^\alpha\geq 1$ such that $$r(B) \sim \dist(B,0)^{1-\alpha},$$ where $r(B)$ denotes the radius of $B$. The proof of Theorem \ref{MultipliersSparse} under these weaker assumptions is conceptually not harder, and only requires a reelaboration of the $L^r-L^s$ bounds for $T_m^j$ and the decay bounds for $T_m^{j,l}$. This may be done using similar ideas to those in the proof of Lemma \ref{iilemma}. The details are left to the interested reader.
\newline
\indent
The above \textit{subdyadic} Hörmander condition was introduced by Bennett and the first author in \cite{BB}, to which we refer for further discussion. Unlike in there, or in the classical Mikhlin--Hörmander theorem, our methods seem to require the conditions \eqref{MikhlinCondition} and \eqref{HormSD} to hold for all $\gamma \in \N^n$ and not only for those $|\gamma| \leq \lfloor \frac{n}{2} \rfloor +1$. It would be interesting to determine if Theorem \ref{MultipliersSparse} holds under these weaker assumptions in the number of derivatives.

\appendix

\section{A pointwise sparse bound for the classes $S^{-n(1-\rho)}_{\rho,\delta}$}
Given a sublinear operator $T$, let $\CM_T$ be its \textit{grand maximal function}, defined by
\begin{equation*}\label{GrandMaximalFunction}
\CM_T f(x):= \sup_{Q \ni x} \esssup_{z \in Q} |T(f\chi_{\R^n \backslash 3Q})(z)|;
\end{equation*}
here the supremum is taken over all cubes $Q \subset \R^n$ containing $x$. In \cite{LeNew}, Lerner proved the following abstract theorem for pointwise sparse domination, which relies on the boundedness of the operator $\CM_T$.
\begin{theorem}[\cite{LeNew}]\label{GeneralPointwise}
Assume that $T$ is a sublinear operator of weak type $(q,q)$ and $\CM_T$ is of weak type $(r,r)$, where $1 \leq q \leq r < \infty$. Then, for every compactly supported $f \in L^r(\R^n)$, there exists a sparse family $\CS$ such that for a.e. $x \in \R^n$,
$$
|Tf(x)| \leq C \CA_{r,\CS} f(x),
$$ 
where $C=C(n,q,r)(\|T\|_{L^q \to L^{q,\infty}}+ \|\CM_T\|_{L^r \to L^{r,\infty}})$.
\end{theorem}

As is well known that $T_a$ is of weak type $(1,1)$ for $a \in S^{-n(1-\rho)}_{\rho,\delta}$, the above Theorem of Lerner reduces the proof of Theorem \ref{pointwise} to verifying that $\CM_{T_a}$ is of weak type $(r,r)$ for any $r>1$. The proof of this fact is quite standard, and is based on work of Chanillo and Torchinksy \cite{CT}, and Michalowski, Rule and Staubach \cite{MRScan}.

\begin{proof}[Proof of Theorem \ref{pointwise}]

Given a point $x$ and a cube $Q \ni x$, we distinguish two cases, $|Q|\leq 1$ and $|Q|>1$. The latter case is easy to deal with, and using decay estimates of the associated kernel to $T_a$, it is elementary to see that
\begin{equation}\label{Q less 1}
|T_a(f \chi_{\R^n \backslash 3Q})(z)|  \lesssim Mf(x)
\end{equation}
for any $z \in Q$; see, for instance \cite{MRScan}.
\newline
\indent
If $|Q|\leq 1$, a slightly more refined analysis is required. For any $z, x' \in Q$, 
\begin{align}
|T_a(f \chi_{\R^n \backslash 3Q})(z)| & \leq |T_a(f \chi_{\R^n \backslash 3Q})(z)-T_a(f \chi_{\R^n \backslash 3Q})(x')| \notag \\ 
& \;\;\;\; + |T_a(f )(x')| + |T_a (f \chi_{3Q})(x')|. \label{break pointwise}
\end{align}
In \cite{MRScan} (see also \cite{CT}), it is proven that 
$$
|T_a(f \chi_{\R^n \backslash 3Q})(z)-T_a(f \chi_{\R^n \backslash 3Q})(x')| \lesssim M_pf(x)
$$
for any $p:=p(\rho)>1$ sufficiently close to 1 and $Q \ni x$ with $|Q|\leq 1$. Raising the estimate \eqref{break pointwise} to a power $1<s<r$, integrating with respect to $x' \in Q$, and raising it again to the power $1/s$,
\begin{align}
|T_a(f \chi_{\R \backslash 3Q})(z)| \! & \lesssim \! M_pf(x) \! + \! \Big(\frac{1}{|Q|} \! \int_Q \! |T_a f(x')|^s dx' \!\Big)^{1/s} \! \!\!\! \notag + \! \Big(\frac{1}{|Q|}\!  \int_Q \!\! |T_a (f \chi_{3Q})(x')|^s dx' \!\Big)^{1/s} \notag \\
& \lesssim \! M_pf(x) \! + \! M_s(T_a f)(x) + \|T_a\|_{s} \Big(\frac{1}{|Q|}\int_{3Q} | f(x')|^s dx' \Big)^{1/s} \notag \\
& \lesssim \! M_pf(x) \! + \! M_s(T_a f)(x) + \|T_a\|_{s}M_sf (x), \label{Q more than 1}
\end{align} 
after taking supremum over all $Q \ni x$ and using the boundedness of $T_a$ on $L^s$. Combining \eqref{Q less 1} and \eqref{Q more than 1}
$$
\CM_{T_a} f(x) \lesssim Mf(x) + M_pf(x) + M_s(T_a f)(x) + \|T_a\|_{s} M_s f(x).
$$
Taking $L^r$-norms, with $r>\max(p,s)$,
$$
\|\CM_{T_a} f\|_r \lesssim \big( \|M\|_r + \|M\|_{r/p} +\|M\|_{r/s}\|T_a\|_r + \|T_a\|_s \|M\|_{r/s} \big) \|f\|_r.
$$
As $p$ and $s$ may be chosen arbitrarily close to $1$, $\CM_{T_a}$ is bounded on $L^r$ for any $r>1$. Thus, an application of Theorem \ref{GeneralPointwise} yields
\begin{equation*}
|T_a f(x)| \lesssim A_{r,\CS} f(x),
\end{equation*}
as required.
\end{proof}

\bibliographystyle{abbrv} 

\bibliography{SparsePseudosRevised}

\end{document}